\documentclass[reqno, oneside, 12pt]{amsart}

\usepackage[letterpaper]{geometry}
\usepackage{amsmath}
\usepackage{amssymb}
\usepackage{enumerate}
\usepackage{enumitem}
\usepackage{verbatim}
\usepackage{mathrsfs}
\usepackage{mathtools}
\usepackage{stmaryrd}

\usepackage{graphicx,amssymb, amsmath, amsthm, cite, xcolor}
\usepackage{fullpage, comment,mathrsfs}

\usepackage{tikz}
\usetikzlibrary{patterns}
\newcommand{\scale}{1.4}

\mathtoolsset{centercolon}

\usepackage{setspace}

\numberwithin{equation}{section}

\newtheorem{theorem}{Theorem}
\newtheorem{corollary}[theorem]{Corollary}

\newtheorem*{lemma*}{Lemma}
\newtheorem{proposition}[theorem]{Proposition}

\theoremstyle{remark}
\newtheorem*{remark*}{Remark}

\renewcommand{\tilde}{\widetilde}

\newcommand{\R}{\mathbb{R}}

\newcommand{\N}{\mathbb{N}}
\newcommand{\C}{\mathbb{C}}

\newcommand{\ds}{\displaystyle}
\newcommand{\mb}[1]{\mathbf{#1}}
\newcommand{\field}[1]{\mathbb{#1}}

\date{}
\author{Thomas E Carty}
\address{Department of Mathematics\\
Bradley University\\
Peoria, IL 61625}
\email{tcarty@bradley.edu}

\title{Elementary Solutions for a model Boltzmann Equation in one-dimension and the connection to Grossly Determined Solutions}

\begin{document}
\begin{abstract}
The Fourier-transformed version of the BGK model in one-dimension is solved in order to determine the general solution's asymptotics.  The ultimate goal of this paper is to demonstrate that the solution to the model Boltzmann possesses a special property that was conjectured by Truesdell and Muncaster: that solutions decay to a subclass of the solution set uniquely determined by the initial first moment of the gas. First we determine the spectrum and eigendistributions of the associated homogeneous equation.  Then, using Case's method of elementary solutions, we find analytic time-dependent solutions to the original problem.  In doing so, we show that the spectrum separates the solutions into two distinct parts; one that behaves as a set of transient solutions and the other limiting to a stable subclass of solutions.  This demonstrates that in time all gas flows for the one-dimensional BGK model Boltzmann act as grossly determined solutions.
 \end{abstract}
 \thanks{The author acknowledges support
from National Science Foundation grant DMS 08-38434 ``EMSW21-MCTP: Research
Experience for Graduate Students" and from the Caterpillar Fellowship Grant at Bradley University.}
\maketitle

\section{Introduction}

The partial integro-differential equation (PIDE) \eqref{thepide} is a simplification of the linearized Boltzmann equation in one-dimension.  Let $x\in\field{R}$ represent the position of a molecule and let $v\in \field{R}$ be the velocity of that molecule.  We consider the model of fluid motion dictated by
\begin{eqnarray}
\label{thepide}
\frac{\partial f}{\partial t}(t,x,v)+v\frac{\partial f}{\partial x}(t,x,v)=-f(t,x,v)+\int_{-\infty}^{\infty} \phi(w)f(t,x,w)dw
\end{eqnarray}
where $\phi(w)$ is the probability density function $ \phi(v)=e^{-v^2}/\sqrt{\pi}$ and the unknown function $f(t,x,v)$ represents the molecular density function of a monotomic gas.  In the classical Boltzmann theory, the right hand side of \eqref{thepide} is interpreted as the collision operator -- dictating the behavior of the gas under inter-molecular collisions.  A purely one-dimensional derivation of \eqref{thepide} can be found in \cite{carty:2016}.  The equation can also be interpreted as the BGK model of the Boltzmann equation (see \cite{BGK(1954)} and Chapter IV in \cite{Cerci(1969)MMKT_text}) under the simplification that velocity is no longer allowed to wander 3-dimensionally.

In \cite[Ch. XXIII]{Tru_n_Mun}, C. Truesdell and R. G. Muncaster remark that -- no matter which model of gas flow you begin with -- the ultimate goal is the same: determine the density, velocity and temperature fields of the gas.  They then note that many of the known exact solutions of Boltzmann's equation -- such as those solutions derived from Hilbert's iteration (see \cite[pg. 316]{Cerci_Dilute_Text} or \cite[Ch. XXII]{Tru_n_Mun}), or the Chapman and Enskog procedure (see \cite[pg. 86]{harris2012introduction}) -- shared the property that the solution class could be represented as being dependent on one (or more) of the gas's physical properties (moments).  This led them to define the concept of a \textbf{\emph{grossly determined solution}}: a solution which is determined at any given instant by the gross conditions (mass density, velocity, temperature) of the gas at that time.  In their epilogue, the authors suggest that these concepts may lead to a new way forward for finding unifying solutions to the Boltzmann equation:
\begin{enumerate}
	\item In general, can we determine a set of conservation laws that define the gross field properties?
	\item Can we use these conservation laws to determine the class of grossly determined solutions to the problem?
	\item If one could find the class of general solutions, can we show that the general solutions evolve asymptotically in time to the class of grossly determined solutions?
\end{enumerate}
In addition to finding a new, richer class of solutions to the Boltzmann equation (a microscopic/atomic level model of gas flow), the class of grossly determined solutions would now be in terms akin to the solutions of the Navier-Stokes equations (a macroscopic/gross fields model of gas dynamics).

The main result in \cite{carty:2016} was to demonstrate that grossly determined solutions exist to \eqref{thepide} that depend solely on the density field of the gas. In doing so, the first two conjectures were shown true for the model Boltzmann \eqref{thepide}.  Specifically, it was shown that given the density of the gas at any moment of time, one can determine the density function $\rho(t,x)$ of the gas for all time and subsequently a solution $f(t,x,v)$ to \eqref{thepide}, which itself is a function $\rho(t,x)$.  It should be noted that a grossly determined solution based upon the first moment alone is the best that can be expected the model equation \eqref{thepide}.  As opposed to the full Boltzmann equation, the collision operator $C(f):=-f(t,x,v)+\int_{\R} \phi(w)f(t,x,w)dw$ only possesses a conservation of mass condition.  In the reduction of the model, conservation of momentum and energy have been lost in the collision operator.  Thus, we don't expect to find grossly determined solutions based upon velocity and energy.  The goal of this paper is to demonstrate the third item, that general solutions evolve asymptotically in time to the class of grossly determined solutions found in \cite{carty:2016}.  To this end, there are three main components of this paper: construction of the solution candidate via spectral methods, proving that the candidate solution is complete and unique, and demonstrating the desired decay condition.

Solutions to \eqref{thepide} under the Laplace transform were first indicated by Cercignani \cite{Cerci(1962)} (additionally see Chapter VII of \cite{Cerci(1969)MMKT_text} or \cite{Gosse(2012)} (pg. 289-292)) using Case's method of elementary solutions \cite{case1960elementary}.  While in many ways emulating the same technique as Cercignani, there are some distinct differences in this current work.  For one, we will be constructing solutions under the Fourier transform.  This is due to the fact that the construction of the grossly determined solutions obtained in \cite{carty:2016} was done under the Fourier transform. As a result, the spectrum and eigensolutions are of a different form than in Cericignani's work.  The spectral representation of the solution (in both Cercignani's work and this work) results in an integral equation of the third kind where the integral is over $\R$.  However, in this work we will show there is an additional portion of the continuous spectrum for the operator.  In fact, it is exactly this new real-valued line segment in the spectrum that corresponds to the subclass of grossly determined solutions and guarantees the asymptotic behavior in Truesdell and Muncaster's third conjecture.  We will also show that the familiar portion of the spectrum, corresponding to the infinite line resulting in the integral equation of the third kind, behaves as a set of transient solutions.  In the end, we show that in time all gas flows for the model Boltzmann \eqref{thepide} act as grossly determined solutions.

The majority of the applications of Case's method are to linear transport equations and result in integral equations over the domain $[-1,1]$ where special care needs to be taken at the end points.  (See \cite{klinc1975eigenfunctions, Hangelbroek, Spec_Mthds_Linear_Trans}.)  For application of elementary solutions to model Boltzmann equations, the boundary conditions are often either overlooked (see \cite{Cerci(1962),cercignani1971methods}) or addressed by a change of variables transformation so the integral is over $[-1,1]$ (see \cite{degroot_dalitz1997exact, dalitz1997half}) and solved using more modern integral equation methods (see \cite{Estrada}).  We desire a clean connection between grossly determined solutions and the general solution class.  To this end, we appeal to a result of Gakhov \cite{Gakhov3rd} that allows us to solve the associated Riemann-Hilbert problem from the natural integral equation that arises from integrating over the infinite spectrum.  In the end, we will have an exact solution to the model Boltzmann \eqref{thepide} that is dependent upon initial data.

\section{Associated Spectral Problem}

We begin by taking the Fourier transform of equation \eqref{thepide} in the spatial variable.  This yields the transformed PIDE
\begin{equation*}
\frac{\partial \hat{f}}{\partial t}(t,\xi,v)=-v i\xi\hat{f}(t,\xi,v)-\hat{f}(t,\xi,v)+\int_{\R} \phi(w)\hat{f}(t,\xi,w)dw
\end{equation*}
and we write this as
\[\frac{\partial \hat{f}}{\partial t}(t,\xi,v)=L(\hat{f})(t,\xi,v)\]
where the operator $L$ is defined to be \begin{eqnarray}
\label{L}
L(g)(\xi,v) &:=& -\xi v i g(\xi,v)-g(\xi,v)+\int_{\R} \phi(w)g(\xi,w)dw.
\end{eqnarray}

We need to consider an appropriate class of functions on which $L$ will operate.  In the classical theory on the Boltzmann equation it is tradition to work with a function space defined by a weighted $L_2$ norm.  Doing so imbues the operator on the right-hand side of the Boltzmann equation (the collision operator) with the desired properties of being self-adjoint and negative definite. (See \cite{Dolera}.)  In \cite{carty:2016}, the linear operator $C(g):= -g(\xi,v)+\int_{\R}\phi(w)g(\xi,w)\,dw$ is shown to be, for each fixed $\xi$, self-adjoint and semi-negative definite on the function space $\mathscr{F}_v$ defined to be the class of functions such that
\[\|f(v)\|_{2,\phi}^2=\int_{\R}|f(v)|^2\phi(v)dv  < \infty \text{ where } \phi(v)=e^{-v^2}/\sqrt{\pi} .\]

Note that $L$ can be written in terms of the collision operator,
\[L(g) = -\xi v i g +C(g).\]
We want to guarantee that repeated applications of $L$ to a function from our chosen function space will remain in our function space.  To this end, we require that our functions are Schwartz class in both the spatial and velocity variables. Thus we define $\mathscr{F}$ such that
\[\|f(\xi,v)\|_{2,\phi}^2=\int_{\R^2}|f(\xi,v)|^2\phi(v)dv d\xi < \infty\]
and such that $\lim_{\xi,v \to \pm \infty}\xi^n v^m f(\xi,v) =0$ for all $n,m\in\N$.
By construction, the operator $L$ maps functions from $\mathscr{F}$ to $\mathscr{F}$ and repeated applications of the operator $L$ to a function $f(\xi,v) \in \mathscr{F}$ will remain in the space of functions.

\section{The Resolvent Operator}

In order to determine the potential spectrum of $L$, we begin by identifying the resolvent operator.  Upon constructing the spectral decomposition, we are likely to find that we need even more that just operating over the functions space $\mathscr{F}$.  However, at this time we can formally find $(L-\lambda)^{-1}$ and use its construction to hint at where the spectrum of $L$ must lie.

\begin{proposition}
Let $h\in \mathscr{F}$ and $\lambda \in \C$ and consider the equation $(L-\lambda)g=h$.  Then the formal inverse of $(L-\lambda)$ is defined by
\begin{equation}
(L-\lambda)^{-1}h(\xi,v):=\frac{-1}{1 + \lambda+\xi v i}\left(\ds h(\xi,v)+\frac{\ds \int_{\R}\frac{h(\xi,v)\phi(v)dv}{1 + \lambda+\xi v i}}{\ds 1-\int_{\R}\frac{\phi(v)dv}{1 + \lambda+\xi v i}}\right).  \label{resolvent_op}
\end{equation}
\end{proposition}

\begin{proof}
We begin with $(L-\lambda)g=h$ and formally solve for $g$:
\[-v i\xi g(\xi,v)-g(\xi,v)+\int_{\R} \phi(w)g(\xi,w)dw-\lambda g(\xi,v) = h(\xi,v).\]
Provided that $1+\lambda+\xi v i\neq 0$, we find that
\[g(\xi,v)=-\frac{1}{1+\lambda+\xi v i}\left(\ds h(\xi,v)-\int_{\R} \phi(w)g(\xi,w)dw\right).\]
We now seek an expression for $ \int_{R} \phi(w)g(\xi,w)dw$, independent of $g(\xi,v)$, that we can use to substitute into the last equation.  Multiplying by $\phi(v)$ and integrating with respect to $v$ yields
\[-\int_{\R}\phi(v)g(\xi,v)dv=\int_{\R} \frac{\ds \phi(v)h(\xi,v)dv}{1+\lambda+\xi v i}-\left(\int_{\R} \phi(w)g(\xi,w)dw\right)\left(\int_{\R} \frac{\phi(v)dv}{1+\lambda+\xi v i}\right).\]
Now, provided that $ 1-\int_{\R} \frac{\phi(v)dv}{1+\lambda+\xi v i}\neq 0$, we obtain the following expression:
\[\int_{\R}\phi(w)g(\xi,w)dw=-\frac{\ds \ds \int_{\R}\frac{h(\xi,v)\phi(v)dv}{1 + \lambda+\xi v i}}{\left(\ds 1-\int_{\R} \frac{\phi(v)dv}{1+\lambda+\xi v i}\right)}.\]
\end{proof}

The computation in the previous proof suggests where the spectrum may be located.  We look for spectral values to occur for $\lambda$ such that either
\[1-\int_{\R} \frac{\phi(v)dv}{1+\lambda+\xi v i}= 0 \text{ or } 1+\lambda+\xi v i = 0.\]
Upon determining where the spectrum must lie, we will return to the formal inverse of $(L-\lambda)$ and demonstrate that we have captured the resolvent set.

\emph{Remark:} In this instance, the computation of \eqref{resolvent_op} was straight-forward enough to do directly.  The same could have been accomplished via the theory of rank-one perturbations of self-adjoint operators.  From this viewpoint, the spectrum of a rank-one perturbation can be determined by the spectrum of the unperturbed operator and information about the perturbation determinant $d=1-\int_{\R} \frac{\phi(v)dv}{1+\lambda+\xi v i}$.  (See \cite{simon1995spectral} or \cite{liaw2015rank1notes}).  Here equation \eqref{resolvent_op} is closely related to the Aronszjan-Krein formula.  More specifically, Equation \eqref{resolvent_op} is exactly Equation (2.3) in the lecture notes of Liaw and Treil \cite{liaw2015rank1notes} using the operator $L$.

\section{The Spectral Decomposition}

We search for a basis (set of generalized eigenfunctions) $B(\xi,v)$ such that $L(B)=\lambda B$.  It is important to note that in this spectral problem, we interpret it as a spectral problem in $v$ dependent on a parameter $\xi$.  We are not computing the spectral problem in 2-dimensions simultaneously.  Using the definition of $L$, the eigenvalue problem is
\begin{equation}
\label{ss_eigen_equation}
-i \xi v B(\xi,v)-B(\xi,v)+\int_{\R} \phi(w)B(\xi,w)dw = \lambda B(\xi,v).
\end{equation}
We can then derive a recursive representation of the basis function
\begin{equation}
\label{recursive_B}
B(\xi,v)=\frac{\ds \int_{\R} \phi(w)B(\xi,w)dw}{1 +\lambda+\xi v i}.
\end{equation}

Define
\begin{equation}
\label{little b}
b(\xi):=\int_{\R} \phi(w) B(\xi,w) dw.
\end{equation}
Now \eqref{recursive_B} becomes \begin{equation}
\label{recursive_b}
B(\xi,v)=\frac{b(\xi)}{1 +\lambda+\xi v i}.
\end{equation}  Multiplying by $\phi(v)$ and integrating over the velocity space yields
\[
b(\xi)=b(\xi) \int_{\R}\frac{\phi(v)dv}{1+\lambda+\xi v i}.
\]  The above computation results in a constraint equation for $\lambda$ dependent on variable $\xi$.

\begin{proposition} Let $B(\xi,v)$ satisfy the spectral equation $L(B)=\lambda B$ and let \[ b(\xi)=\int_{\R} \phi(w) B(\xi,w) dw.\]
For all $\xi$ for which $b(\xi)\neq 0$, $\lambda$ must satisfy the constraint
\begin{equation}
\label{lambda_constraint_eqn}
\int_{\R}\frac{\phi(v)dv}{1+\lambda+\xi v i}=1.
\end{equation}
\end{proposition}

\subsection{The Spectrum: The Real-Valued Portion}

\subsubsection{The interval $\mb (-1,0)$}

\begin{proposition}
\label{the_real_spectrum} The interval $(-1,0)\subseteq \R$ is part of the spectrum.
\end{proposition}
\begin{proof}
Assume $\lambda$ is real-valued.  Begin with the left-hand side of constraint equation \eqref{lambda_constraint_eqn}.
\begin{align*}
    \int_{\R}\frac{\phi(v)dv}{1+\lambda+\xi v i} &= \int_{\R}\frac{\phi(v)[(1+\lambda)-(\xi v)i]dv}{(1+\lambda)^2+(\xi v)^2}\\
    \ &= (1+\lambda)\int_{\R}\frac{\phi(v)dv}{(1+\lambda)^2+(\xi v)^2}-\xi i \int_{\R}\frac{v\phi(v)dv}{(1+\lambda)^2+(\xi v)^2}\\
    \ &= (1+\lambda)\int_{\R}\frac{\phi(v)dv}{(1+\lambda)^2+(\xi v)^2} \tag{since the right-hand integrand is odd in  $v$.}
\end{align*}
Therefore, when $\lambda$ is real, the constraint equation for $\lambda$ reduces to \begin{equation}
\label{constraint real}
\int_{\R}\frac{(1+\lambda)\phi(v)dv}{(1+\lambda)^2+(\xi v)^2}=1.
\end{equation}  Note that this equation requires that $\lambda\neq -1$.

In \cite{carty:2016}, the odd invertible function $\Xi:\R\backslash 0 \rightarrow  (-\sqrt{\pi},0)\cup(0,\sqrt{\pi})$ defined
\begin{equation*}
\Xi(\eta):=\int_{\R} \frac{\eta\phi(v)dv}{\eta^2+v^2}
\end{equation*}
was analyzed. (See Figure \ref{xi versus eta}.)

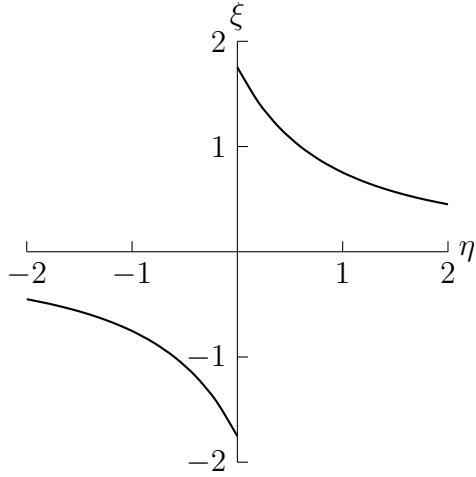
\begin{figure}
  \centering
\begin{tikzpicture}[>=stealth]
\draw (-2*\scale,0*\scale) -- (2*\scale,0*\scale) node[anchor=west]{$\eta$};
\draw (0*\scale,-2*\scale) -- (0*\scale,2*\scale) node[anchor=south]{$\xi$};
\draw (.1*\scale,2*\scale) -- (0*\scale,2*\scale) node[anchor=east]{$2$};
\draw (.1*\scale,1*\scale) -- (0*\scale,1*\scale) node[anchor=east]{$1$};
\draw (.1*\scale,-2*\scale) -- (0*\scale,-2*\scale) node[anchor=east]{$-2$};
\draw (.1*\scale,-1*\scale) -- (0*\scale,-1*\scale) node[anchor=east]{$-1$};
\draw (2*\scale,.1*\scale) -- (2*\scale,0*\scale) node[anchor=north]{$2$};
\draw (1*\scale,.1*\scale) -- (1*\scale,0*\scale) node[anchor=north]{$1$};
\draw (-2*\scale,.1*\scale) -- (-2*\scale,0*\scale) node[anchor=north]{$-2$};
\draw (-1*\scale,.1*\scale) -- (-1*\scale,0*\scale) node[anchor=north]{$-1$};
\draw[thick] plot[smooth] coordinates{(0*\scale, 1.75263*\scale) (0.2*\scale, 1.4198*\scale) (0.4*\scale, 1.17853*\scale) (0.6*\scale, 0.998537*\scale) (0.8*\scale, 0.860816*\scale) (1*\scale, 0.753057*\scale) (1.2*\scale, 0.667063*\scale) (1.4*\scale, 0.597234*\scale) (1.6*\scale, 0.539653*\scale) (1.8*\scale, 0.491519*\scale) (2*\scale, 0.450792*\scale)};
\draw[thick] plot[smooth] coordinates{(0*\scale, -1.75263*\scale) (-0.2*\scale, -1.4198*\scale) (-0.4*\scale, -1.17853*\scale) (-0.6*\scale, -0.998537*\scale) (-0.8*\scale, -0.860816*\scale) (-1*\scale, -0.753057*\scale) (-1.2*\scale, -0.667063*\scale) (-1.4*\scale, -0.597234*\scale) (-1.6*\scale, -0.539653*\scale) (-1.8*\scale, -0.491519*\scale) (-2*\scale, -0.450792*\scale)};
\end{tikzpicture}
\caption{the graph of $ \xi=\Xi(\eta)$} \label{xi versus eta}
\end{figure}

Provided $\xi\neq 0$, we see that $\eqref{constraint real}$ can be rewritten as
\begin{eqnarray*}
    \int_{\R}\frac{\left(\frac{1+\lambda}{\xi}\right)\phi(v)dv}{\left(\frac{1+\lambda}{\xi}\right)^2+v^2}&=& \xi, \text{ or }\\
    \Xi\left(\frac{1+\lambda}{\xi}\right)&=& \xi.
\end{eqnarray*}
This suggests that we should seek to use the the inverse of $\Xi(\eta)$ to rewrite $\lambda$ in terms of $\eta$, and hence $\xi$.
Let $\eta(\xi)=\Xi^{-1}(\xi)$.
Using \eqref{constraint real}, we equate $\eta=\dfrac{1+\lambda}{\xi}$ and find a parametric representation for $\lambda=\Lambda(\xi)$ where $\Lambda(\xi):=-1+\xi \eta(\xi)$, the graph of which is Figure \ref{xi versus lambda}.

\begin{figure}
  \centering
  \includegraphics[width=3in]{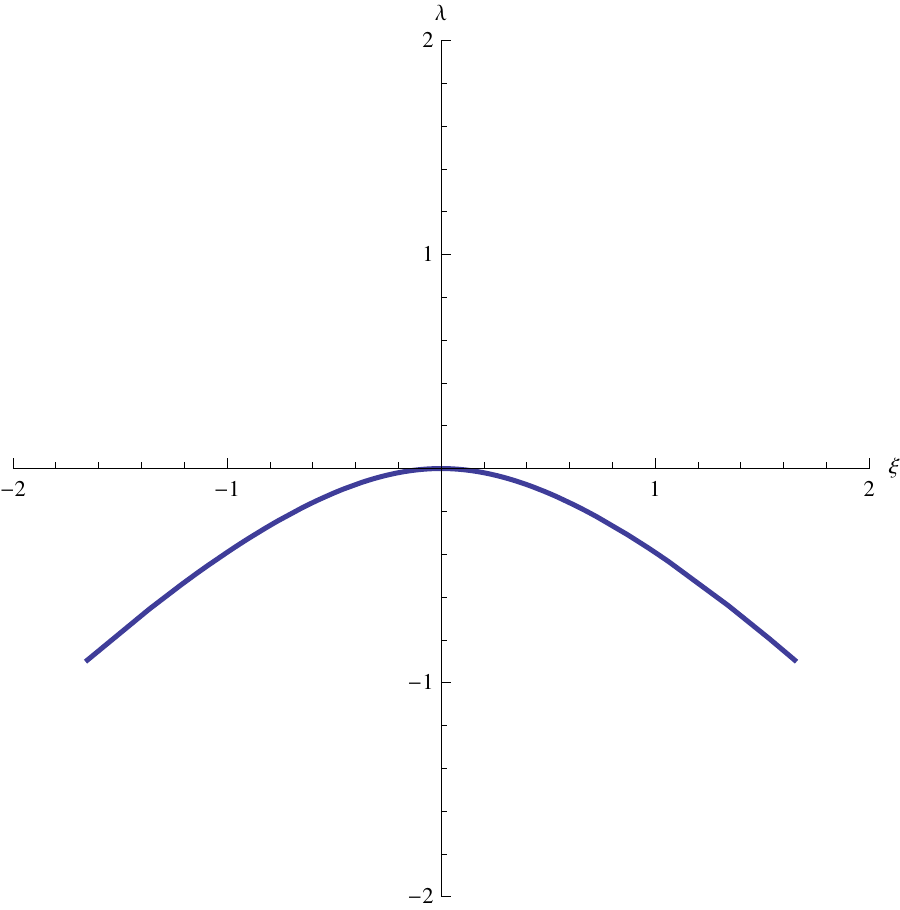}\\
  \caption{the graph of $\lambda=\Lambda(\xi)$} \label{xi versus lambda}
\end{figure}

To formally analyse the range of values $\lambda$ can take and still satisfy the identity \eqref{constraint real}, we look to the range of the function $\Lambda(\xi)$.  Given that $\eta\rightarrow 0^+$, as $\xi \rightarrow \sqrt{\pi}^-$, we see that $\lim_{\xi\rightarrow \sqrt{\pi}^-}(-1+\xi \eta(\xi))=-1$.  To determine the behavior of $\lambda$ as $\xi$ goes to 0, we examine the behavior of $\xi \eta(\xi)$.  Here, it is easier to look at the inverse problem and examine $\Xi(\eta)\eta$.  We have that \[\Xi(\eta) \eta = \left(\int_{\R} \frac{\eta\phi(v)dv}{\eta^2+v^2}\right)=\int_{\R} \frac{\eta^2\phi(v)dv}{\eta^2+v^2}.\]  Hence,
\begin{eqnarray*}
    \lim_{\eta\rightarrow \infty}\Xi(\eta) \eta &=& \lim_{\eta\rightarrow \infty} \int_{\R} \frac{\eta^2\phi(v)dv}{\eta^2+v^2},\\
    \ &=& \lim_{\eta\rightarrow \infty} \int_{\R} \frac{(\eta^2+v^2-v^2)\phi(v)dv}{\eta^2+v^2},\\
    \ &=& \int_{\R} \phi(v)dv -\lim_{\eta\rightarrow \infty} \int_{\R} \frac{v^2\phi(v)dv}{\eta^2+v^2},\\
    \ &=& 1-0.
\end{eqnarray*}  Thus, $\lim_{\eta\rightarrow \infty}(-1+\Xi(\eta) \eta)=0$.  This demonstrates that the real-valued portion of the spectrum is the interval $(-1,0)$.

Since $\eta(\xi)$ is an odd, invertible function, analysis of $-1+\Xi(\eta) \eta$ as $\eta\rightarrow 0^-$ and $\eta\rightarrow -\infty$ results in the same portion of the real axis.
\end{proof}

\subsubsection{Associated Eigendistributions}

At this moment, we appear to have a collection of spectral values that are dependent upon the parameter $\xi$.  We now will free $\lambda$ of its parametric dependence and construct functions akin to eigensolutions.

We know that $\lambda$ and $\xi$ are related by
\begin{equation}
\label{constraint real2}1- \int_{\R}\frac{\phi(v)dv}{1 +\lambda+\xi v i}=0
\end{equation}
and this can be rewritten as $\ds \xi=\Xi\left(\frac{1+\lambda}{\xi}\right).$  This means that $\lambda=-1+\xi \eta(\xi)=\Lambda(\xi)$ where we have a nice graph of $\Lambda$, Figure \ref{xi versus lambda}.  Thus, for any $\xi$ in the domain of $\Lambda$: \begin{equation}
\label{constraint real with Lambda}\int_{\R}\frac{\phi(v)dv}{1 +\Lambda(\xi)+\xi v i}=1
\end{equation} and \eqref{constraint real2} is not valid if $\lambda$ does not have the form $\Lambda(\xi)$ for some $\xi$.

We can see from the graph of $\Lambda$ that for each fixed value of $\lambda$ in the interval $(-1,0)$ there corresponds two values of $\xi$.  In turn, we can identify these corresponding values of $\xi$ with $\lambda$.  Let $E^+(\lambda)$ denote the inverse of $\Lambda(\xi)$ over the restricted domain $\xi\in(0,\sqrt{\pi})$ and let $E^-(\lambda)$ denote the inverse of $\Lambda(\xi)$ over $\xi\in(-\sqrt{\pi},0)$.  In this notation, for each fixed $\lambda$ in $(-1,0)$, $E^+(\lambda)$ and $E^-(\lambda)$ are the corresponding values of $\xi$.  Thus \eqref{constraint real with Lambda} can be rephrased in terms of $\lambda$ rather than $\xi$ as
\begin{equation}
\label{constraints real in lambda}
\int_{\R}\frac{\phi(v)dv}{1 +\lambda+E^+(\lambda) v i}=1 \text{ and } \int_{\R}\frac{\phi(v)dv}{1 +\lambda+E^-(\lambda)  v i}=1
\end{equation}
and \eqref{constraint real2} is not valid if $\xi$ does not have the form $E^+(\lambda)$ or $E^-(\lambda)$ for some $\lambda$ in $(-1,0)$.

We now use these facts and build eigendistributions.  Let $B_{\lambda}$ be the basis function dependent on $\lambda$.  Beginning with \eqref{recursive_B}, the recursive definition of $B(\xi,v)$, we use the definition of $b_{\lambda}(\xi)$ to derive  \begin{equation}
\label{eigenfunction constraint equation}
b_{\lambda}(\xi)\left(1- \int_{\R}\frac{\phi(v)dv}{1 +\lambda+\xi v i}\right)=0.
\end{equation}
From \eqref{constraints real in lambda} we see that $b_{\lambda}(\xi)$ must vanish at all points $\xi$ other than $E^+(\lambda)$ and $E^-(\lambda)$.  Additionally, the value of $b_{\lambda}(\xi)$ is arbitrary at these two points.  This leads us to the two distributional solutions
\begin{equation}
\label{signed b_lambda}
b_{\lambda}^+=\delta(\xi-E^+(\lambda)) \text{ and } b_{\lambda}^-=\delta(\xi-E^-(\lambda)),
\end{equation}
where $\delta(x)$ denotes the Dirac delta distribution.
In general, $b_{\lambda}$ will be an arbitrary superposition of these two distributions.  By \eqref{recursive_B}, we obtain two eigendistributions for each $\lambda$ in $(-1,0)$:
\begin{equation}
\label{real_eigendists}
B_{\lambda}^+(\xi,v)=\frac{\delta(\xi-E^+(\lambda))}{1 +\lambda+\xi v i}\text{ and } B_{\lambda}^-(\xi,v)\frac{\delta(\xi-E^-(\lambda))}{1 +\lambda+\xi v i}
\end{equation}

Our ultimate goal is to create a transform.  That is, we will seek to integrate in $\lambda$ over the interval $(-1,0)$ an arbitrary superposition of the basis functions of the form \[ C^+(\lambda,\xi)\frac{\delta(\xi-E^+(\lambda))}{1 +\lambda+\xi v i}+C^-(\lambda,\xi)\frac{\delta(\xi-E^-(\lambda))}{1 +\lambda+\xi v i}.\]  As posed, the integration requires a change of variables so that the delta distributions are of the form $\delta(\lambda-\Gamma)$ where $\Gamma$ is independent of $\lambda$.  Note that all the action of $\delta(\xi-E^+(\lambda))$ occurs at $\xi=E^+(\lambda)$.  Since $\Lambda(\xi)$ is the inverse of $E^+(\lambda)$, we can equivalently say that all of the action of $\delta(\xi-E^+(\lambda))$ occurs at $\Lambda(\xi)=\lambda$.  This suggests that we redefine $b_{\lambda}(\xi)=\delta(\lambda-\Lambda(\xi))$.  Note that we are not claiming that $\delta(\xi-E^+(\lambda))=\delta(\lambda-\Lambda(\xi))$.  In fact, the required change of variables would induce the Jacobian $\dfrac{1}{|\Lambda'(E^+(\lambda))|}$.  However, since the Jacobian is now purely a function of $\lambda$, it can be subsumed by the arbitrary coefficient function.  Therefore we can define $b_{\lambda}^+(\xi)=\delta(\lambda-\Lambda(\xi))$.  Since $\Lambda(\xi)$ is also the inverse of $E^-(\lambda)$, the same argument concludes that $b_{\lambda}^-(\xi)=\delta(\lambda-\Lambda(\xi))$.  In essence, we obtain just one solution $b_{\lambda}(\xi)=\delta(\lambda-\Lambda(\xi))$ and the one eigendistribution \begin{equation}
\label{real_eigendist}
B_{\lambda}(\xi,v)=\frac{\delta(\lambda-\Lambda(\xi))}{1 +\lambda+\xi v i}.
\end{equation}

A somewhat simpler way to arrive at the same conclusion is to think of \eqref{eigenfunction constraint equation} in a different way.  Rather than viewing this as a problem for a function of $\xi$ depending parametrically on $\lambda$, we can view it as a problem of solving for a function of $\lambda$ depending parametrically on $\xi$.  Given our initial calculations, we conclude immediately that $b_{\lambda}(\xi)$ will be some multiple of $\delta(\lambda-\Lambda(\xi))$ and this leads immediately to \eqref{real_eigendist}.

\begin{theorem}
\label{basis1 thm}
Define $L$ as in \eqref{L}.  Then the distribution $B_{\lambda}$ defined by \eqref{real_eigendist} satisfies the equation $L(B_{\lambda})=\lambda B_{\lambda}$ for each $\lambda \in (-1,0)$.
\end{theorem}
\begin{proof}
Notice that
\begin{eqnarray*}
L(B_{\lambda})(\xi,v)&=& L\left(\frac{\delta(\lambda-\Lambda(\xi))}{1 +\lambda+\xi v i}\right)\\
 \ &=& -\xi v i \left(\frac{\delta(\lambda-\Lambda(\xi))}{1 +\lambda+\xi v i}\right)-\left(\frac{\delta(\lambda-\Lambda(\xi))}{1 +\lambda+\xi v i}\right)+\int_{\R} \frac{\phi(w)\delta(\lambda-\Lambda(\xi))dw}{1 +\lambda+\xi w i}\\
\ &=& (-1-\xi v i )\left(\frac{\delta(\lambda-\Lambda(\xi))}{1 +\lambda+\xi v i}\right)+ \delta(\lambda-\Lambda(\xi))\int_{\R} \frac{\phi(w)dw}{1 +\lambda+\xi w i}.
\end{eqnarray*}
The last integral was analysed in the proof of Proposition \ref{the_real_spectrum}.  Moreover, in the derivation of $B_{\lambda}$, it was shown that when $\lambda = \Lambda(\xi)$, the integral is identically 1.  Hence, \begin{align*}
L(B_{\lambda})(\xi,v) &=  (-1-\xi v i )\left(\frac{\delta(\lambda-\Lambda(\xi))}{1 +\lambda+\xi v i}\right)+ \delta(\lambda-\Lambda(\xi))\\
&= \lambda\left(\frac{\delta(\lambda-\Lambda(\xi))}{1 +\lambda+\xi v i}\right)\\
&= \lambda B_{\lambda}(\xi,v). \qedhere
\end{align*}
\end{proof}

\subsection{The Spectrum: The Complex-Valued Portion}

\subsubsection{The line $\boldsymbol{\ell}:=-1+\alpha i$}

The other divisibility condition required in the derivation of the resolvent operator \eqref{resolvent_op} is that $1+\lambda+\xi v i \neq 0$.  This suggest that we should consider spectral values of the form $\lambda = -1- \xi v i$.  Since $\xi$ and $v$ are free variables over $\R$, this continuum of points is more concisely written $\lambda=-1+\alpha i$, $\alpha \in \R$.  We define $\boldsymbol{\ell}$ to be this line in the complex plane.

\subsubsection{Associated Eigendistributions}

We begin again with \eqref{ss_eigen_equation} and let $\lambda \in \boldsymbol{\ell}$.  Then \eqref{recursive_B}, the recursive form of $B(\xi,v)$, becomes  \begin{equation}
\label{recursive_B_v2}
B(\xi,v)=\frac{\ds \int_{\R} \phi(w)B(\xi,w)dw}{(\alpha+\xi v) i}.
\end{equation}
Previously, we converted this equation to an equation in $b(\xi)$ and made use of the properties of the delta distribution.  This time we don't have that luxury.  Multiplying \eqref{recursive_B_v2} by $\phi(v)$ and integrating over $v$-space yields

\[b(\xi)\left(1-\int_{\R}\frac{dv}{(\alpha+\xi v) i}\right)=0.\]  The only solution to this equation is $b(\xi)=0$.

That said, our previous work indicated that we should be searching for a distributional basis function.  Mimicking the technique suggested by Case \cite{case1960elementary} and Cercignani \cite{Cerci(1962),cercignani1966method}, we append to \eqref{recursive_B_v2} the weighted Dirac mass whose action occurs at the singularity caused by $\alpha+\xi v$.  That is,
\begin{align*}
B(\xi,v)&= \frac{\ds \int_{\R} \phi(w)B(\xi,w)dw}{(\alpha+\xi v) i} + K(\xi)\delta(\xi v+\alpha).
\end{align*}
where $K(\xi)$ is a function that will allow this basis candidate to satisfy the definition of $b(\xi)$.

Using the definition of $b(\xi)$, we solve for $K(\xi)$ explicitly. \begin{align*}
b(\xi)&= \int_{\R}\phi(w)B(\xi,w)dw \tag{by definition}\\
\ &= b(\xi)\left(\int_{\R}\frac{\phi(w)dw}{(\alpha +\xi w)i}\right)+\int_{\R} \phi(w)K(\xi)\delta(\xi w+\alpha)\,dw.
\end{align*}  To make sense of the resultant integrals, we need to view them distributionally.  In particular, the first integral is viewed as a Cauchy principal value integral (denoted p.v.). Hence, \begin{align*}
b(\xi)& =b(\xi)\left(\textup{p.v.}\int_{\R}\frac{\phi(w)dw}{(\alpha +\xi w)i}\right)+\int_{\R} \phi(w) K(\xi)\delta(\xi w+\alpha)\,dw\\
b(\xi) &= b(\xi)\left(\textup{p.v.}\int_{\R}\frac{\phi(w)}{i(\xi w +\alpha)}\,dw\right) + K(\xi)\int_{\R}\phi(\beta/\xi)\delta(\beta+\alpha)\,\frac{d\beta}{|\xi|} \tag{where $\beta=\xi v$}\\
b(\xi)&= b(\xi)\left(\textup{p.v.}\int_{\R}\frac{\phi(w)}{i(\xi w +\alpha)}\,dw\right) + K(\xi)\frac{\phi(-\alpha/\xi)}{|\xi|}\\
K(\xi)&= \frac{|\xi| b(\xi)}{\phi(-\alpha/\xi)}\left(1-\textup{p.v.}\int_{\R}\frac{\phi(w)}{i(\xi w +\alpha)}\,dw\right)
\end{align*}
Using the derivation of $K(\xi)$, $B(\xi,v)$ becomes \[B(\xi,v)= \frac{b(\xi)}{i(\xi v +\alpha)} + \frac{|\xi| b(\xi)}{\phi(-\alpha/\xi)}\left(1-\textup{p.v.}\int_{\R}\frac{\phi(w)}{i(\xi w +\alpha)}\,dw\right)\delta(\xi v+\alpha).\]  Note that each term of $B(\xi,v)$ is being multiplied by $b(\xi)$.  As we did before, we can let this be subsumed into the functional coefficient in the transform.  (This has the added benefit of normalizing $B$ with respect to mass-density.)  In other words, for this definition of $B(\xi,v)$, $b(\xi)=1$.   Hence, for each $\lambda\in \boldsymbol{\ell}$, we define the associated eigendistribution
\begin{equation}
\label{complex_eigendist}
B_{\lambda(\alpha)}(\xi,v)=\frac{1}{i(\xi v +\alpha)} + \frac{|\xi|}{\phi(-\alpha/\xi)}\left(1-\textup{p.v.}\int_{\R}\frac{\phi(w)}{i(\xi w +\alpha)}\,dw\right)\delta(\xi v+\alpha).
\end{equation}

\begin{theorem}
\label{basis2 thm}
Define $L$ as in \eqref{L}.  Then for each $\lambda=-1+\alpha i \in \boldsymbol{\ell}$, the distribution $B_{\lambda(\alpha)}$ \eqref{complex_eigendist} satisfies the equation $L(B_{\lambda})=\lambda B_{\lambda}$.
\end{theorem}

\begin{proof}
This is a straight-forward computation using the machinery built above: \begin{align*}
L(B_{\lambda(\alpha)}) &= -\xi v i B_{\lambda(\alpha)}(\xi,v)-B_{\lambda(\alpha)}(\xi,v)+\int_{\R}\phi(w)B_{\lambda(\alpha)}(\xi,w)\,dw\\
&= -B_{\lambda(\alpha)}(\xi,v) -\xi v i B_{\lambda(\alpha)}(\xi,v)+1 \tag{ since  $b(\xi)=1$}\\
&= -B_{\lambda(\alpha)}(\xi,v) +\alpha i B_{\lambda(\alpha)}(\xi,v)- \alpha i B_{\lambda(\alpha)}(\xi,v)-\xi v i B_{\lambda(\alpha)}(\xi,v)+1\\
&= (-1+\alpha i) B_{\lambda(\alpha)}(\xi,v)-(\alpha + \xi v)i B_{\lambda(\alpha)}(\xi,v)+1\\
&= (-1+\alpha i) B_{\lambda(\alpha)}(\xi,v)+\left[-1-K(\xi)(\alpha+\xi v)\delta(\xi v+\alpha)\right]+1\\
&= (-1+\alpha i) B_{\lambda(\alpha)}(\xi,v)+[-1+0]+1 \tag{ since  $x\delta(x)=0$}\\
&= (-1+\alpha i) B_{\lambda(\alpha)}(\xi,v). \qedhere
\end{align*}

\end{proof}

\subsubsection{The Spectral Value $\lambda=0$}

We have identified the line $\boldsymbol{\ell}$ and the interval $(-1,0)$ as belonging to the spectrum of $L$.  For a ``nice enough" operator, we would expect the complement of the resolvent set to be closed \cite{Kress}.  However, currently the union of our spectral pieces are not a closed set.  We will close $\boldsymbol{\ell}\cup (-1,0)$ by demonstrating that $\lambda=0$ is a spectral value.

Recall that by construction of the PIDE, $\ds \int_{\R}\phi(v)dv=1$.  By \eqref{constraint real}, when $\lambda=0$ we have \[\int_{\R}\dfrac{\phi(v)dv}{1+\xi^2v^2}=1.\]  It is clear that this identity will only hold when $\xi=0$.  Now consider the graph of $\Lambda(\xi)$, Figure \ref{xi versus lambda}.  This suggests that we should be able to redefine $\Lambda(\xi)$ continuously by including the point $(0,0)$.  With this extension, equation \eqref{eigenfunction constraint equation} becomes \[b_0(\xi)\left[1-\int_{\R}\dfrac{\phi(v)dv}{1+\xi^2v^2}\right]=0\] and we get the distributional solution $b_0(\xi)=\delta(\lambda - \Lambda(0))$.  Now the eigendistribution \eqref{real_eigendist} can be extended to a basis function on the half-open interval $\lambda\in(-1,0]$.  Hence, we have closed the spectrum.

\subsection{The Resolvent Set}

We now take a moment to show that we have in fact captured the spectrum of the operator $L$.

\begin{theorem}
Define $\mb{S}=\boldsymbol{\ell}\cup(-1,0]$.  For $\lambda \in \C/\mb{S}$, the operator $(L-\lambda)^{-1}$ over $\mathscr{F}$ defined by \eqref{resolvent_op} is bounded.  In other words, $\C/\mb{S}$ is in the resolvent set of $L$.
\end{theorem}

\begin{proof}
Fix $\lambda \in \C/\mb{S}$.  Define $d$ to be the distance between the point $\lambda$ and the line $\boldsymbol{\ell}$; $d=\text{dist}(\lambda,\boldsymbol{\ell})$.  Additionally, for $\lambda\not\in (-1,0]$, the proof of Prop \ref{the_real_spectrum} shows that
\[\int_{\R}\frac{(1+\lambda)\phi(v)dv}{(1+\lambda)^2+(\xi v)^2}\neq 1.\]
Hence, we can define the constant $\gamma$ to be
\[\gamma = 1-\int_{\R}\frac{(1+\lambda)\phi(v)dv}{(1+\lambda)^2+(\xi v)^2}.\]
Then,
\begin{align*}
\|(L-\lambda)^{-1}h(\xi,v)\|_{2,\phi} &= \left\| \frac{-1}{1 + \lambda+\xi v i}\left(\ds h(\xi,v)+\frac{\ds \int_{\R}\frac{h(\xi,v)\phi(v)dv}{1 + \lambda+\xi v i}}{\ds 1-\int_{\R}\frac{\phi(v)dv}{1 + \lambda+\xi v i}}\right)\right\|_{2,\phi}\\
&\leq \dfrac{1}{d}\left(\|h(\xi,v)\|_{2,\phi}+\dfrac{1}{|\gamma|}\left\|\ds \int_{\R}\frac{h(\xi,v)\phi(v)dv}{1 + \lambda+\xi v i}\right\|_{2,\phi}\right)\\
&\leq \dfrac{1}{d}\left(\|h(\xi,v)\|_{2,\phi}+\dfrac{1}{|\gamma| d}\left\|\ds \int_{\R} h(\xi,v)\phi(v)dv\right\|_{2,\phi}\right)\\
&\leq \dfrac{1}{d}\left(1+\dfrac{1}{|\gamma| d}\right)\|h(\xi,v)\|_{2,\phi}
\end{align*}
Hence $\mb{S}$ is the continuous spectrum of $L$.
\end{proof}

\subsection{The Transform Candidate}

Again letting $\mb{S}=\boldsymbol{\ell}\cup(-1,0]$, we seek to represent functions in the form \[\hat{f}(\xi,v)=\int_{\mb{S}}C(\lambda,\xi)B_{\lambda}(\xi,v)\,d\lambda.\]
By the preceding, we have \begin{equation}
\label{spec candidate} \hat{f}(\xi,v)=\int_{(-1,0]}C(\lambda,\xi)B_{\lambda}(\xi,v)\,d\lambda+\int_{\boldsymbol{\ell}}K(\lambda(\alpha),\xi)B_{\lambda(\alpha)}(\xi,v)\,d\alpha.
\end{equation}  We now examine these two integrals independently.

For $\lambda\in(-1,0]$, the integration is straight-forward as our basis is a delta distribution in $\lambda$: \begin{align}
\nonumber
\int_{(-1,0]}C(\lambda,\xi)B_{\lambda}(\xi,v)\,d\lambda &= \int_{(-1,0]}C(\lambda,\xi)\frac{\delta(\lambda-\Lambda(\xi))}{1 +\lambda+\xi v i}\,d\lambda \tag{by \eqref{real_eigendist}}\\
\nonumber &= \dfrac{C(\Lambda(\xi),\xi)}{1 +\Lambda(\xi)+\xi v i}\\
\label{simplified basis1}&= \dfrac{C_\Lambda(\xi)}{1 +\Lambda(\xi)+\xi v i}
\end{align}
where $C_\Lambda(\xi):=C(\Lambda(\xi),\xi)$.

For $\lambda\in\boldsymbol{\ell}$, the integration is a little more technical.  We already have a parametrization for $\boldsymbol{\ell}$ in terms of $\alpha$ and $B_{\lambda(\alpha)}$ in terms of the same parameter.  So \begin{align*}
\int_{\boldsymbol{\ell}}K(\lambda(\alpha),\xi)&B_{\lambda(\alpha)}(\xi,v)\,d\alpha\\
 &= \int_{\R}\tilde{K}(\alpha,\xi)B_{\lambda(\alpha)}(\xi,v)\,d\alpha \tag{where $\tilde{K}(\alpha,\xi):=K(-1+\alpha i, \xi)$}\\
 &= \int_{\R}\tilde{K}(\alpha,\xi)\left[\frac{1}{i(\xi v +\alpha)} + \frac{|\xi|}{\phi(-\alpha/\xi)}\left(1-\textup{p.v.}\int_{\R}\frac{\phi(w)}{i(\xi w +\alpha)}\,dw\right)\delta(\xi v+\alpha)\right]\,d\alpha
 \tag{by \eqref{complex_eigendist}}\\
&= \text{p.v.} \int_{\R}\frac{\tilde{K}(\alpha,\xi)d\alpha}{i(\xi v +\alpha)} + \frac{|\xi|}{\phi(v)}\left(1-\textup{p.v.}\int_{\R}\frac{\phi(w)}{i(\xi w -\xi v)}\,dw\right)\tilde{K}(-\xi v, \xi).
\end{align*}
In the first integral, make the change of variables $\alpha=-\xi w$ to get \begin{align}
\nonumber \int_{\boldsymbol{\ell}}K(\lambda(\alpha),\xi)&B_{\lambda(\alpha)}(\xi,v)\,d\alpha\\
\nonumber &= \text{p.v.} \int_{\R}\frac{\tilde{K}(-\xi w,\xi)|\xi|dw}{i(\xi v -\xi w)} + \frac{|\xi|}{\phi(v)}\left(1-\textup{p.v.}\int_{\R}\frac{\phi(w)}{i(\xi w -\xi v)}\,dw\right)\tilde{K}(-\xi v, \xi)\\
\label{simplified basis2}
&= -\text{p.v.} \int_{\R}\frac{K_{\xi}(w)dw}{i(w-v)} + \frac{1}{\phi(v)}\left(\xi-\textup{p.v.}\int_{\R}\frac{\phi(w)}{i( w - v)}\,dw\right)K_{\xi}( v)
\end{align}
where $\ds K_{\xi}(v):=\frac{|\xi|}{\xi}\tilde{K}(-\xi v,\xi)$.
Therefore, the integral transform associated with the spectral decomposition is \begin{equation}
\label{transform}
\hat{f}(\xi,v)=\dfrac{C_\Lambda(\xi)}{1 +\Lambda(\xi)+\xi v i}-\text{p.v.} \int_{\R}\frac{K_{\xi}(w)dw}{i(w-v)} + \frac{1}{\phi(v)}\left(\xi-\textup{p.v.}\int_{\R}\frac{\phi(w)}{i( w - v)}\,dw\right)K_{\xi}( v)
\end{equation}

\subsubsection{Applying the Operator to the Transform Candidate}

Ultimately our goal is to use our spectral decomposition to solve for the general solution to $\hat{f}_t(t,\xi,v)=L\hat{f}$.  It will be important to have a representation of $L\hat{f}$ in terms of the spectral coefficients $C_\Lambda(\xi)$ and $K_{\xi}( v)$.  Since the computation of $L\hat{f}$ closely mirrors the computations above, we will prove the action of $L$ on our decomposition here.

\begin{corollary}
\label{cor L on basis1}
Let $L$ be defined by \eqref{L} and $B_{\lambda}$ defined by \eqref{real_eigendist}.  Then \begin{equation*}
L\left(\int_{(-1,0]}C(\lambda,\xi)B_{\lambda}(\xi,v)\,d\lambda\right) = \Lambda(\xi) \dfrac{C_\Lambda(\xi)}{1 +\Lambda(\xi)+\xi v i}.
\end{equation*}
\end{corollary}
\begin{proof}
\begin{align*}
L\left(\int_{(-1,0]}C(\lambda,\xi)B_{\lambda}(\xi,v)\,d\lambda\right) &= \int_{(-1,0]}C(\lambda,\xi)L(B_{\lambda}(\xi,v))\,d\lambda \\
&=  \int_{(-1,0]}C(\lambda,\xi)\lambda B_{\lambda}(\xi,v))\,d\lambda \tag{by Theorem \ref{basis1 thm}}\\
&= \int_{(-1,0]}C(\lambda,\xi)\lambda \frac{\delta(\lambda-\Lambda(\xi))}{1 +\lambda+\xi v i}\,d\lambda \tag{by \eqref{real_eigendist}}\\
&= \Lambda(\xi) \dfrac{C_\Lambda(\xi)}{1 +\Lambda(\xi)+\xi v i} \tag{by computation similar to \eqref{simplified basis1}}
\end{align*}
\end{proof}

\begin{corollary}
\label{cor L on basis2}
Let $L$ be defined by \eqref{L} and  $B_{\lambda(\alpha)}$ is defined by \eqref{complex_eigendist}.  Then \begin{align*}
L&\left(\int_{\boldsymbol{\ell}}K(\lambda(\alpha),\xi)B_{\lambda(\alpha)}(\xi,v)\,d\alpha\right)\\
 &= -\textup{p.v.} \int_{\R}\frac{(-1-\xi w i)K_{\xi}(w)dw}{i(w-v)} + \frac{1}{\phi(v)}\left(\xi-\textup{p.v.}\int_{\R}\frac{\phi(w)}{i( w - v)}\,dw\right)(-1-\xi v i)K_{\xi}(v).
\end{align*}
\end{corollary}
\begin{proof}
\begin{align*}
L\left(\int_{\boldsymbol{\ell}}K(\lambda(\alpha),\xi)B_{\lambda(\alpha)}(\xi,v)\,d\alpha\right) &= \int_{\boldsymbol{\ell}}K(\lambda(\alpha),\xi)L\left(B_{\lambda(\alpha)}(\xi,v)\right)\,d\lambda\\
&=  \int_{\boldsymbol{\ell}}K(\lambda(\alpha),\xi)(-1+\alpha i) B_{\lambda(\alpha)}(\xi,v))\,d\lambda \tag{by Theorem \ref{basis2 thm}}\\
&= -\int_{\boldsymbol{\ell}}K(\lambda(\alpha),\xi)B_{\lambda(\alpha)}(\xi,v))\,d\lambda+i\int_{\boldsymbol{\ell}}K(\lambda(\alpha),\xi)\alpha B_{\lambda(\alpha)}(\xi,v))\,d\lambda
\end{align*}
The first integral has been computed above as \eqref{simplified basis2}.  Similarly,
\begin{align*}
i \int_{\boldsymbol{\ell}}K(\lambda(\alpha),\xi)&\alpha B_{\lambda(\alpha)}(\xi,v)\,d\alpha\\
 &= -\text{p.v.} \int_{\R}\frac{(\xi w i)K_{\xi}(w)dw}{i(w-v)} + \frac{1}{\phi(v)}\left(\xi-\textup{p.v.}\int_{\R}\frac{\phi(w)}{i( w - v)}\,dw\right)(-\xi v i)K_{\xi}(v).
\end{align*}
Summing these two integral yields the desired result.
\end{proof}

\subsection{Solving the Singular Integral Equation}

We begin with the observation that the principal value integrals in \eqref{transform} are multiples of the Hilbert transforms of $K_{\xi}(w)$ and $\phi(w)$, respectively.  Numerically, the Hilbert transform of $e^{-x^2}$ is well understood in terms of the Dawson function $D(y)$ where $\ds D(y)=e^{-y^2}\int_0^y e^{x^2}\,dx$ \cite{Dawson}, see Figure \ref{Dawson fcn}.
\begin{figure}
  \centering
  \includegraphics[width=3in]{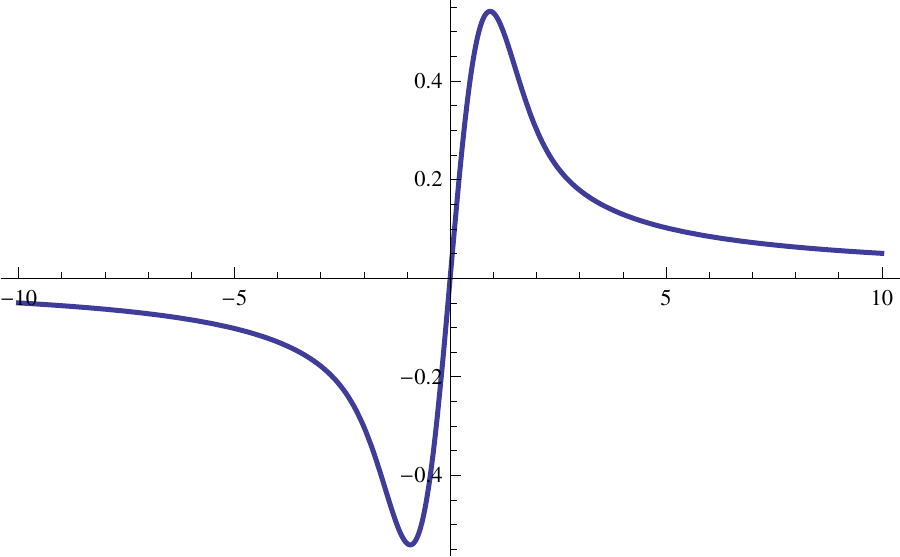}\\
  \caption{The graph of the Dawson function} \label{Dawson fcn}
\end{figure}
Specifically: \[\textup{p.v.}\int_{\R}\frac{\phi(w)}{i( w - v)}\,dw=2iD(v).\]  We now manipulate \eqref{transform} into the standard form for a Carleman type singular integral equation \cite{Estrada}:
\begin{align*}
\hat{f}(\xi,v)-\dfrac{C_\Lambda(\xi)}{1 +\Lambda(\xi)+\xi v i} &= -\text{p.v.} \int_{\R}\frac{K_{\xi}(w)dw}{i(w-v)} + \frac{1}{\phi(v)}\left(\xi-2iD(v)\right)K_{\xi}( v),
\end{align*}
\begin{align}
\label{SIE1}
\phi(v)\left(\hat{f}(\xi,v)-\dfrac{C_\Lambda(\xi)}{1 +\Lambda(\xi)+\xi v i}\right) &= \left(\xi-2iD(v)\right)K_{\xi}( v)+\dfrac{-\pi \phi(v)}{\pi i}\text{p.v.} \int_{\R}\frac{K_{\xi}(w)dw}{w-v}.
\end{align}
In order to simplify the following analysis, we make the following notational changes:
\begin{equation}
\label{SIE2}
F_{\xi}(v)=A_{\xi}(v)K_{\xi}(v)+\dfrac{B(v)}{\pi i}\text{p.v.} \int_{\R}\frac{K_{\xi}(w)dw}{w-v},
\end{equation}
where $A_{\xi}(v):= \xi-2iD(v)$, $B(v) := -\pi \phi(v)$, and
\begin{align*}
F_{\xi}(v) &:= \phi(v)\left(\hat{f}(\xi,v)-\dfrac{C_\Lambda(\xi)}{1 +\Lambda(\xi)+\xi v i}\right).
\end{align*}

By the definitions of $\phi(v)$ and the Dawson function $D(v)$, we know that $A_{\xi}(v)$ and $B(v)$ are H\"{o}lder continuous functions in $v$.  We will require $\hat{f}(\xi,v)$ be such that $F_{\xi}(v)$ is H\"{o}lder continuous as well.

\subsubsection{Reduction to an Associated Riemann Problem}

Solving \eqref{SIE2} requires converting the equation into its equivalent Riemann problem \cite{Gakhov2nd}, \cite{Musk}.  Using the Sokhotski-Plemelj formulas, an equivalent representation of this problem is to seek a sectionally analytic function $\Phi_{\xi}(v)$ satisfying the boundary condition \begin{equation}
\label{Riemann}
\Phi_{\xi}^+(v) = G_{\xi}(v)\Phi_{\xi}^-(v)+g_{\xi}(v)
\end{equation} where $G_{\xi}(v)=\dfrac{A_{\xi}(v)-B(v)}{A_{\xi}(v)+B(v)}$ and $g_{\xi}(v)=\dfrac{F_{\xi}(v)}{A_{\xi}(v)+B(v)}$  on the real-axis in the complex plane.  Solvability of this Riemann problem begins with the conditions that $G_{\xi}(v)$ and $g_{\xi}(v)$ are defined and non-vanishing on the real-line.  Equivalently, we need $A_{\xi}(v)-B(v)\neq 0$ and $A_{\xi}(v)+B(v)\neq 0$.  Note that $A_{\xi}(v)-B(v)=\xi+\pi\phi(v)-2iD(v)$.  On $(-\infty, \infty)$, $D(v)$ vanishes only at $v=0$.  When $v=0$, $A_{\xi}(0)-B(0)=\xi+\pi\phi(0)=\xi+\sqrt{\pi}$.  Hence $A_{\xi}(v)-B(v)$ vanishes when $v=0$ and $\xi=-\sqrt{\pi}$.  An equivalent computation shows that $A_{\xi}(v)+B(v)$ vanishes when $v=0$ and $\xi=\sqrt{\pi}$.  For the remainder of this discussion, we will assume that the parameter $\xi\neq \pm\sqrt{\pi}$.

In the classical theory, the solution(s) to the Riemann problem \begin{equation}
\label{Riemann general}
\Phi^+(z) = G(z)\Phi^-(z)+g(z)
\end{equation} are constructed for boundary problems with finite simple (often closed) boundary curve $\mb{\gamma}$ in the complex plane.  The representation of the problem's solution is dependent on the problem's index.  Define \begin{equation}
\label{chi}
\chi= \text{Ind }G(z)=\dfrac{1}{2\pi}[\text{arg }G(z)]_{\mb{\gamma}}
\end{equation} where $[\circ ]_{\mb{\gamma}}$ denotes the increment of the expression in the brackets as the result of one traversal along $\mb{\gamma}$.  In other words, $\chi$ is the winding number of the image of the boundary curve $\mb{\gamma}$ in $\C$ under the map $G(z)$.
Let $X(z)$ be the solution to the associated homogeneous problem
\[X^+(z) = G(z)X^-(z).\]
When $\chi\geq 0$, the general solution is a certain particular solution of \eqref{Riemann general} plus a summation of $\chi$ linearly independent solutions of the homogeneous problem.  In particular, the solution is
\[\Phi(z)=\dfrac{X(z)}{2\pi i}\int_{\mb{\gamma}}\dfrac{g(\tau)}{X^+(\tau)}\dfrac{d\tau}{\tau-z}+X(z)P_{\chi}(z)\]
where $P_{\chi}(z)$ is a polynomial of degree $\chi$ with arbitrary coefficients.  If we impose the additional constraint that $\Phi(z)$ decay to 0 at infinity, then the general solution has the same form except the polynomial must be of one degree less.  Here the solution is
\[\Phi(z)=\dfrac{X(z)}{2\pi i}\int_{\mb{\gamma}}\dfrac{g(\tau)}{X^+(\tau)}\dfrac{d\tau}{\tau-z}+X(z)P_{\chi-1}(z).\]
If $\chi<0$, the solution -- with decay at infinity -- takes the same form except now $P_{\chi}(z)\equiv 0$.

In Gakhov's untranslated \emph{Boundary Value Problems}, 3rd edition \cite[pg 191]{Gakhov3rd} we find the following result for Riemann problems with infinite boundary.  For completeness, we have included the theorem in entirety.
\begin{theorem}\footnote{Many thanks to Vadim Zharnitsky for aid in the translation.}
The singular equation \begin{equation}
\label{gen Carleman}
A(t)\phi(t)+\dfrac{B(t)}{\pi i}\int_{\R}\dfrac{\phi(\tau)}{\tau-t}d\tau=f(t)
\end{equation} and the Riemann problem for the real line with the extra condition \begin{equation}
\label{constraint at infinity}
2A(\infty)c=f(\infty)B(\infty)
\end{equation}
are equivalent in the following sense: if $\Phi(z)$ is a general solution for the boundary problem
\begin{equation}
\label{Hilbert}
\Phi^+(t)=\dfrac{A(t)-B(t)}{A(t)+B(t)}\Phi^-(t)+\dfrac{f(t)}{A(t)+B(t)} \tag{$-\infty< t < \infty$}
\end{equation}
satisfying the condition \eqref{constraint at infinity}, where $c$ is the leading coefficient for the polynomial $P_{\chi}(z)$ for $\chi\geq0$, and $\ds c=-\dfrac{1}{2\pi i}\int_{\R}\dfrac{f(\tau)}{[A(\tau)+B(\tau)]X^+(\tau)}\dfrac{d\tau}{\tau+i}$ for $\chi<0$, then the function $\phi(t)=\Phi^+(t)-\Phi^-(t)$ is the solution of \eqref{gen Carleman}.  Conversely, if $\phi(t)$ is the general solution to \eqref{gen Carleman}, then the Cauchy type integral $\ds \Phi(z)=\dfrac{1}{2\pi i}\int_{\R}\dfrac{\phi(\tau)}{\tau-z}d\tau$ is the solution of the Riemann problem \eqref{Hilbert} satisfying the condition \eqref{constraint at infinity}.
\end{theorem}
In other words, provided the constraint \eqref{constraint at infinity} is satisfied, the solution to the Riemann problem for the half-plane can be used to construct the solution for \eqref{SIE2}.

\subsubsection{Index of the Singular Equation}

For our problem, \begin{equation}
\chi= \text{Ind }G_{\xi}(v)=\dfrac{1}{2\pi}[\text{arg }G_{\xi}(v)]_{\R}.
\end{equation}
is the winding number of the image of the boundary curve $\R$ in $\C$ under the map $G_{\xi}(v)$.
Using the definition of $G_{\xi}(v)$, we see that \begin{equation}
\label{G}
G_{\xi}(v)=\dfrac{\xi^2-\pi^2\phi^2(v)+4D^2(v)}{(\xi-\pi\phi(v))^2+4D^2(v)}+\left(\dfrac{4\pi\phi(v)D(v)}{(\xi-\pi\phi(v))^2+4D^2(v)}\right)i.
\end{equation}
Note that the image of $\R$ is real-valued at only three points, $v=0$ and $v=\pm \infty$.  Since $G_{\xi}(\pm\infty)=(1,0)$, every image curve starts and ends at $(1,0)$.  At $v=0$, $\ds G_{\xi}(0)=\left(\dfrac{\xi+\sqrt{\pi}}{\xi-\sqrt{\pi}},0\right)$.  For the image curve to have a non-zero winding number, it is necessary that $\ds \dfrac{\xi+\sqrt{\pi}}{\xi-\sqrt{\pi}}<0$.  We find that for $\xi\in(-\sqrt{\pi},\sqrt{\pi})$, $\chi=-1$.  For all other $\xi$, $\chi=0$.  (For examples, see Figure \ref{index}.)

\begin{figure}

  \centering

  \includegraphics[width=3in]{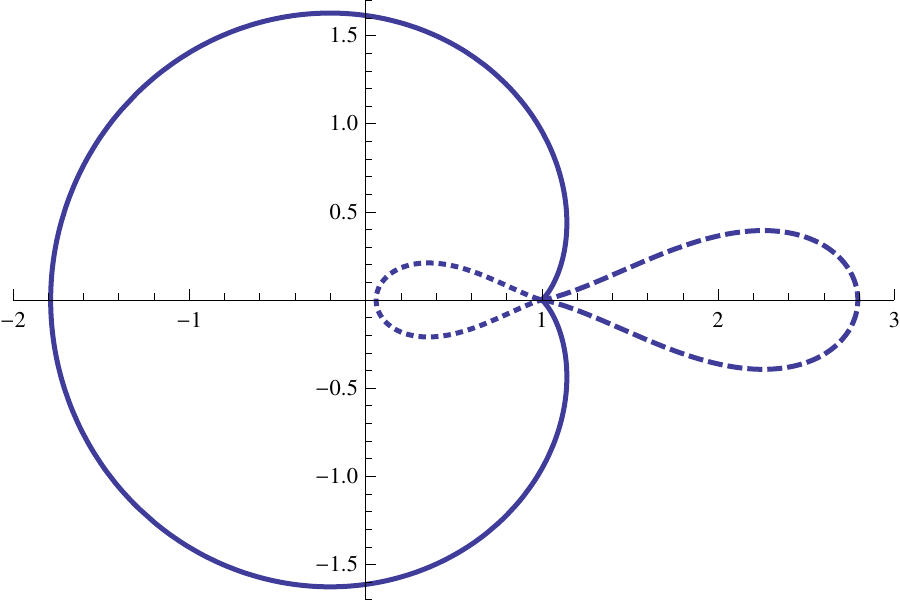}\\
  \caption{Images of the real line under $G_{\xi}(v)$ (solid curve $\xi=1/2$, dotted curve $\xi=-2$, dashed curve $\xi=3.75$)} \label{index}

\end{figure}

\subsubsection{Solution to the Associated Riemann Problem}

Let $\ds X^{+}(v)=e^{\Gamma^+(v)}$ and $\ds X^{-}(v)=\left(\dfrac{v-i}{v+i}\right)^{-\chi}e^{\Gamma^-(v)}$ where \[\Gamma(v)=\dfrac{1}{2\pi i}\int_{\R}\ln \left[\left(\dfrac{\tau-i}{\tau+i}\right)^{-\chi}G_{\xi}(\tau)\right]\dfrac{d\tau}{\tau-v},\] and let \[\Psi(v)=\dfrac{1}{2\pi i}\int_{\R}\dfrac{g_{\xi}(\tau)}{X^+(\tau)}\dfrac{d\tau}{\tau-v}\] where $G_{\xi}(v)$ and $g_{\xi}(v)$ are defined as in \eqref{Riemann}. Then the unique solution to the Riemann problem for the half-plane is given by \[\Phi(v)=X(v)\Psi(v).\]

\subsubsection{Solving for the Coefficients}

Given the above solution, from repeated use of the Sokhotski-Plemelj formulas we get a unique representation for the $K_{\xi}(v)$, namely \begin{align*}
K_{\xi}(v) &=\Phi^+(v)-\Phi^-(v),\\
&= X^+(v)\Psi^+(v)-X^-(v)\Psi^-(v),\\
&=X^+(v)\left(\dfrac{1}{2}\dfrac{g_{\xi}(v)}{X^+(v)}+\dfrac{1}{2\pi i}\int_{\R}\dfrac{g_{\xi}(\tau)}{X^+(\tau)}\dfrac{d\tau}{\tau-v}\right) -X^-(v)\left(-\dfrac{1}{2}\dfrac{g_{\xi}(v)}{X^+(v)}+\dfrac{1}{2\pi i}\int_{\R}\dfrac{g_{\xi}(\tau)}{X^+(\tau)}\dfrac{d\tau}{\tau-v}\right),\\
&= \dfrac{g_{\xi}(v)}{2}\left[1+\dfrac{1}{G_{\xi}(v)}\right]+\dfrac{X^+(v)}{2\pi i}\left[1-\dfrac{1}{G_{\xi}(v)}\right]\int_{\R}\dfrac{g_{\xi}(\tau)}{X^+(\tau)}\dfrac{d\tau}{\tau-v},
\end{align*} since $X^+(v)=G_{\xi}(v)X^-(v)$.
Using the definitions of $X^+(v)$, $G_{\xi}(v)$ and $g_{\xi}(v)$ yields the following representation:
\begin{equation}
\label{coef D}
K_{\xi}(v)=\dfrac{A_{\xi}(v)}{A^2_{\xi}(v)+B^2(v)}F_{\xi}(v)-\dfrac{e^{\Gamma^+(v)} B(v)}{A_{\xi}(v)-B(v)}\dfrac{1}{\pi i}\int_{\R}\dfrac{F_{\xi}(\tau)}{e^{\Gamma^+(\tau)}(A_{\xi}(\tau)+B(\tau))}\dfrac{d\tau}{\tau-v}
\end{equation}
where $A_{\xi}(v)= \xi-2iD(v)$, $B(v) = -\pi \phi(v)$,
\begin{align*}
F_{\xi}(v) &= \phi(v)\left(\hat{f}(\xi,v)-\dfrac{C_\Lambda(\xi)}{1 +\Lambda(\xi)+\xi v i}\right),
\end{align*} and
\begin{align*}
\Gamma^+(v)=\frac{1}{2}\ln \left[\left(\dfrac{v-i}{v+i}\right)^{-\chi}\dfrac{A_{\xi}(v)-B(v)}{A_{\xi}(v)+B(v)}\right]+\dfrac{1}{2\pi i}\int_{\R}\ln\left[\left(\dfrac{\tau-i}{\tau+i}\right)^{-\chi}\dfrac{A_{\xi}(\tau)-B(\tau)}{A_{\xi}(\tau)+B(\tau)}\right]\dfrac{d\tau}{\tau-v}.
\end{align*}  (Recall that $D(v)$ is Dawson's function.)

Additionally, when $\chi=-1$, we get a unique representation of the coefficient $C_{\Lambda}(\xi)$.  Since $B(\infty)=0$ and $A_{\xi}(\infty)=\xi$, the constraint at infinity condition \eqref{constraint at infinity} yields the additional condition that $c\equiv0$, or equivalently \[\int_{\R}\dfrac{F_{\xi}(\tau)}{[A_{\xi}(\tau)+B(\tau)]X^+(\tau)}\dfrac{d\tau}{\tau+i}=0.\] The definition of $F_{\xi}(v)$ yields
\begin{equation}
\label{coef C}
C_{\Lambda}(\xi)=\dfrac{\ds \int_{\R}\dfrac{\phi(\tau)\hat{f}(\xi,\tau)}{e^{\Gamma^+(\tau)}[A_{\xi}(\tau)+B(\tau)][1+\Lambda(\xi)+\xi \tau i]}\dfrac{d\tau}{\tau+i}}{\ds \int_{\R}\dfrac{\phi(\tau)}{e^{\Gamma^+(\tau)}[A_{\xi}(\tau)+B(\tau)][1+\Lambda(\xi)+\xi \tau i]}\dfrac{d\tau}{\tau+i}}
\end{equation}

Now, it is important to note that the form of the coefficients are dependent upon $\chi$ and that $\chi$ is dependent upon $\xi$.  Recall that when $|\xi|\geq \sqrt{\pi}$, $\chi=0$.  Additionally, the value of $C_{\Lambda}(\xi)$ is dependent upon $\Lambda(\xi)$ and $\Lambda(\xi)$ only makes sense for values of $\xi$ in $(-\sqrt{\pi},\sqrt{\pi})$.  In order to make sense of this, in addition to the requirement $\hat{f}(\xi,v)$ be such that $\ds \phi(v)\left(\hat{f}(\xi,v)-\dfrac{C_\Lambda(\xi)}{1 +\Lambda(\xi)+\xi v i}\right)$ is H\"{o}lder continuous in $v$ for all $\xi$ in $(-\sqrt{\pi},\sqrt{\pi})$, we need the additional condition that $\hat{f}(\xi,v)$ be in the class of functions such that $C_{\Lambda}(\xi)\equiv 0$ when $|\xi|\geq \sqrt{\pi}$.  Since we are requiring that $K_{\xi}(v)=0$ at infinity, we are still guaranteed uniqueness of our solution for all $\xi$.

\section{Applying the Spectral Decomposition}

We now apply the properties of our spectral decomposition to find the general solution to the original PIDE \eqref{thepide}.  We have transformed the PIDE into \begin{equation*}
\frac{\partial \hat{f}}{\partial t}(t,\xi,v)=-v i\xi\hat{f}(t,\xi,v)-\hat{f}(t,\xi,v)+\int_{\R} \phi(w)\hat{f}(t,\xi,w)dw
\end{equation*} and we write this as \begin{equation}
\label{trans pde in L}
\frac{\partial \hat{f}}{\partial t}(t,\xi,v)=L(\hat{f})(t,\xi,v)
\end{equation} where $L$ is defined to be \begin{eqnarray*}
L(g)(\xi,v) &=& -\xi v i g(\xi,v)-g(\xi,v)+\int_{\R} \phi(w)g(\xi,w)dw.
\end{eqnarray*}
By the transform arising from our spectral decomposition \eqref{transform}, we look for solutions of the form \begin{equation*}
\hat{f}(t,\xi,v)=\dfrac{C_\Lambda(t,\xi)}{1 +\Lambda(\xi)+\xi v i}-\text{p.v.} \int_{\R}\frac{K_{\xi}(t,w)dw}{i(w-v)} + \frac{1}{\phi(v)}\left(\xi-\textup{p.v.}\int_{\R}\frac{\phi(w)}{i( w - v)}\,dw\right)K_{\xi}(t,v).
\end{equation*}
Then, by Corollaries \ref{cor L on basis1} and \ref{cor L on basis2},  \eqref{trans pde in L} becomes \begin{align*}
 \dfrac{\dfrac{d C_{\Lambda}}{dt}(t,\xi)-\Lambda(\xi)C_\Lambda(t,\xi)}{1 +\Lambda(\xi)+\xi v i}&-\textup{p.v.} \int_{\R}\frac{\left[\dfrac{d K_{\xi}}{dt}(t,w)-(-1-\xi w i)K_{\xi}(t,w)\right]dw}{i(w-v)}\\
&+ \frac{1}{\phi(v)}\left(\xi-\textup{p.v.}\int_{\R}\frac{\phi(w)}{i( w - v)}\,dw\right)\left[\dfrac{d K_{\xi}}{dt}(t,v)-(-1-\xi v i)K_{\xi}(t,v)\right] =0.
\end{align*}
By the uniqueness of our spectral representation, this yields the ODEs
\[\dfrac{d C_{\Lambda}}{dt}(t,\xi)-\Lambda(\xi)C_\Lambda(t,\xi)=0 \text{ and } \dfrac{d K_{\xi}}{dt}(t,v)-(-1-\xi v i)K_{\xi}(t,v)=0.\] Hence,
\[C_\Lambda(t,\xi)=\tilde{C}_{\Lambda}(\xi)e^{\Lambda(\xi) t} \text{ and } K_{\xi}(t,v)=\tilde{K}_{\xi}(v)e^{(-1-\xi v i)t}.\]

\begin{theorem} Let $f(0,x,v):=f_0(x,v)$ represent the initial molecular (number) density of the gas such that $\hat{f}_0(\xi,v)$ is of compact support on $\xi\in(-\sqrt{\pi},\sqrt{\pi})$ and such that \begin{align*}
F_0(\xi,v) &= \phi(v)\left(\hat{f_0}(\xi,v)-\dfrac{C_0(\xi)}{1 +\Lambda(\xi)+\xi v i}\right)
\end{align*}
is H\"{o}lder continuous in $v$ on $\R$ when $C_0(\xi)$ is defined by
\begin{equation*}
C_0(\xi)=\dfrac{\ds \int_{\R}\dfrac{\phi(\tau)\hat{f_0}(\xi,\tau)}{e^{\Gamma^+(\tau)}[A_{\xi}(\tau)+B(\tau)][1+\Lambda(\xi)+\xi \tau i]}\dfrac{d\tau}{\tau+i}}{\ds \int_{\R}\dfrac{\phi(\tau)}{e^{\Gamma^+(\tau)}[A_{\xi}(\tau)+B(\tau)][1+\Lambda(\xi)+\xi \tau i]}\dfrac{d\tau}{\tau+i}}
\tag{when $|\xi|<\sqrt{\pi}$}
\end{equation*}
and $C_0(\xi)\equiv 0$ when $|\xi|\geq \sqrt{\pi}$ where $A_{\xi}(v)= \xi-2iD(v)$, $B(v) = -\pi \phi(v)$, and $\ds D(y)=e^{-y^2}\int_0^y e^{x^2}\,dx$ is Dawson's function.  Then the Fourier transform of the general solution to \eqref{thepide} with initial condition is \begin{align}
\nonumber \hat{f}(t,\xi,& v) = e^{\Lambda(\xi) t}\dfrac{C_0(\xi)}{1 +\Lambda(\xi)+\xi v i}\\
\label{the total solution}
&+e^{-t}\left[-\textup{p.v.} \int_{\R}\frac{e^{(-\xi w i)t}K_0(\xi,w)dw}{i(w-v)} + \frac{e^{(-\xi v i)t}}{\phi(v)}\left(\xi-\textup{p.v.}\int_{\R}\frac{\phi(w)}{i( w - v)}\,dw\right)K_0(\xi,v)\right]
\end{align}
where
\begin{equation*}
K_0(\xi,v)=\dfrac{A_{\xi}(v)}{A^2_{\xi}(v)+B^2(v)}F_0(\xi,v)-\dfrac{e^{\Gamma^+(v)} B(v)}{A_{\xi}(v)-B(v)}\dfrac{1}{\pi i}\int_{\R}\dfrac{F_0(\xi,\tau)}{e^{\Gamma^+(\tau)}(A_{\xi}(\tau)+B(\tau))}\dfrac{d\tau}{\tau-v}
\end{equation*}
where
\begin{align*}
\Gamma^+(v)=\frac{1}{2}\ln \left[\left(\dfrac{v-i}{v+i}\right)^{-\chi}\dfrac{A_{\xi}(v)-B(v)}{A_{\xi}(v)+B(v)}\right]+\dfrac{1}{2\pi i}\int_{\R}\ln\left[\left(\dfrac{\tau-i}{\tau+i}\right)^{-\chi}\dfrac{A_{\xi}(\tau)-B(\tau)}{A_{\xi}(\tau)+B(\tau)}\right]\dfrac{d\tau}{\tau-v}.
\end{align*}

\end{theorem}

\section{General Solutions evolve to Grossly Determined Solutions}

We are  now at a point where we can demonstrate the third conjecture of Truesdell and Muncaster \cite{Tru_n_Mun}: general solutions evolve asymptotically in time to the class of grossly determined solutions.  At this point, we will see that this amounts to nothing more than definition chasing.  What was unexpected apriori is how different portions of the spectrum correspond to the asymptotic behaviour of the solution.  Roughly speaking, the portion of the general solution corresponding to the real-valued part of the spectrum, $(-1,0]$, tends in time specifically to the subclass of solutions defined by the grossly determined solutions; the portion of the general solution corresponding to the vertical line $\boldsymbol{\ell}$ in $\C$ is transient.

Recall that $\Lambda(\xi)$ takes values in the open interval $(-1,0)$ (see Figure \ref{xi versus lambda}).  Hence, the asymptotic behavior of the general solution \eqref{the total solution} tends to the part of the spectral decomposition arising from the real-part of the spectrum.  In other words,
\begin{equation*}
\hat{f}(t,\xi,v)\sim e^{\Lambda(\xi) t}\dfrac{C_0(\xi)}{1 +\Lambda(\xi)+\xi v i} \tag{as $t\rightarrow \infty$}.
\end{equation*}

The main achievement in \cite{carty:2016} was to construct the class of grossly determined solutions to \eqref{thepide}.  By ansatz (motivated by a lemma of H\"{o}rmander \cite[pg 15]{Hormander}), one looked for convolution solutions of the form $f(t,x,v)=\int_{\field{R}}K_v(y)\rho(t,x-y)\,dy$, where $\rho(t,x)$ is the density field.  The ultimate goal of this paper is to show that the any element from the class of general solutions decays (in time) to an element of the subclass of grossly determined solutions.  For completeness, we include the main theorem, Theorem 1, of \cite{carty:2016}.

\begin{theorem}
\label{GDS thm}
Consider the one-dimensional model of fluid flow \begin{equation*}
\frac{\partial f}{\partial t}(t,x,v)+v\frac{\partial f}{\partial x}(t,x,v)=-f(v,x,t)+\int_{-\infty}^{\infty} \phi(w)f(w,x,t)dw \tag{1.1}
\end{equation*}
where $f(t,x,v)$ is the molecular density function of the gas and $\phi$ is the probability density function $\ds \phi(v):=e^{-v^2}/\sqrt{\pi}$.  Let $\rho(t,x)$ represent the density function of the gas:
\[\rho(t,x):=\int_{\field{R}}\phi(v)f(t,x,v)\,dv\]
where the Fourier transform $\hat{\rho}(t,\xi)$ has support within $(-\sqrt{\pi},0)\cup (0,\sqrt{\pi})$. Let $\hat{\rho}_0(\xi)$ denote the Fourier transform of the density function at $t=0$.
Then a solution to equation (\ref{thepide}) is given by
\begin{equation}
f(t,x,v)=\int_{\field{R}}K_v(y)\rho(t,x-y)\,dy.
\end{equation} where the Fourier transform of $f$ is
\begin{equation}
\label{GDSundertransform}
\hat{f}(t,\xi, v)=\left(\frac{1}{1-i\xi k(\xi)+i \xi v }\right)\hat{\rho}_0(\xi)e^{-i \xi k(\xi) t}
\end{equation}
where $\ds k(\xi)=\left(\frac{-1+\xi C(\xi)}{\xi}\right)i$ and $c=C(\xi)$ is defined implicitly by $\ds \xi=\int_{\R} \frac{c\phi(v)}{c^2+v^2}\,dv$.
\end{theorem}

With just a minor bit of manipulation, we can easily see that the portion of the general solution that corresponds to the real part of the spectrum also corresponds to the subclass of grossly determined solutions.  Recall that $\Lambda(\xi)$ was defined $\Lambda(\xi)=-1+\xi C(\xi)$.  Hence, \[-i\xi k(\xi)=-i\xi\left(\frac{-1+\xi C(\xi)}{\xi}\right)i=\Lambda(\xi).\]  Hence, the grossly determined solution (under transform) \eqref{GDSundertransform} can be rewritten as
\begin{equation*}
\hat{f}(t,\xi, v)=e^{\Lambda(\xi) t}\frac{\hat{\rho}_0(\xi)}{1+\Lambda(\xi)+\xi v i}.
\end{equation*}
Thus proving Truesdell and Muncaster's third conjecture for the model Boltzmann \eqref{thepide}: that the subclass of grossly determined solutions acts as an attractor set for the class of general solutions.  That is, in time, all gas flows act as grossly determined solutions.

\section{Conclusions}

In the terms of Truesdell and Muncaster's conjectures on grossly determined solutions, we have demonstrated for the one-dimensional BGK model that the class of general solutions decay asymptotically to the subclass of grossly determined solutions.  In other words, the asymptotic gas flow determined by the BGK model is dictated solely by the initial density field of the gas.  As the BGK model is a linearization of the one-dimensional Boltzmann equation about a maxwellian, one would hope that this paper is a first step in showing that Truesdell and Muncaster's three conjectures also hold for one-dimensional Boltzmann equations with a more robust collision operator.

\bibliographystyle{hplain}
\bibliography{GDS_ElemSolns_bib}

\end{document}